\documentclass[10pt,a4paper]{article}
\usepackage{a4wide}
\usepackage{amsmath,amssymb,amsthm}
\usepackage{tabls}
\usepackage[dvipdfmx]{graphicx}
\usepackage{wrapfig}
\usepackage{array}
\usepackage[algoruled,linesnumbered,algo2e,vlined]{algorithm2e}
\usepackage{url}
\usepackage{multirow}
\usepackage{multicol}
\usepackage{cases}
\usepackage[normalem]{ulem}
\usepackage{arydshln}
\newtheorem{theorem}{Theorem}

\newtheorem{lemma}[theorem]{Lemma}
\newtheorem{corollary}[theorem]{Corollary}

\newcommand{\lynnum}{\mathit{L}}
\newcommand{\occ}{\mathit{Occ}}
\newcommand{\subcontain}{\mathit{C}}
\newcommand{\maxnum}{\mathit{MTS}}
\newcommand{\totalnum}{\mathit{TS}}
\newcommand{\distnum}{\mathit{MDS}}
\newcommand{\exnum}{\mathit{ETS}}
\newcommand{\totaldistinct}{\mathit{TDS}}
\newcommand{\exdistinct}{\mathit{EDS}}
\newcommand{\DF}{\mathit{MDF}}
\newcommand{\TF}{\mathit{MTF}}
\newcommand{\ETF}{\mathit{ETF}}
\newcommand{\EDF}{\mathit{EDF}}
\title{
  Counting Lyndon Subsequences
}
\author{
  Ryo~Hirakawa$^{1}$ \quad
  Yuto~Nakashima$^{2}$ \quad 
  Shunsuke~Inenaga$^{2,3}$ \quad 
  Masayuki~Takeda$^{2}$ \\ 
  \\
  {$^1$ Department of Information Science and Technology,} \\
  {Kyushu University, Fukuoka, Japan}\\
    {\texttt{hirakawa.ryo.460@s.kyushu-u.ac.jp}}\\
  {$^2$ Department of Informatics, Kyushu University, Fukuoka, Japan}\\
    {\texttt{\{yuto.nakashima,inenaga,takeda\}@inf.kyushu-u.ac.jp}}\\
  {$^3$ PRESTO, Japan Science and Technology Agency, Kawaguchi, Japan}\\
}

\begin{document}
\maketitle

\begin{abstract}
  Counting substrings/subsequences that preserve some property (e.g., palindromes, squares) 
  is an important mathematical interest in stringology.
  Recently, Glen et al. studied the number of Lyndon factors in a string.
  A string $w = uv$ is called a Lyndon word
  if it is the lexicographically smallest among all of its conjugates $vu$.
  In this paper, we consider a more general problem "counting Lyndon subsequences".
  We show 
  (1)~the maximum total number of Lyndon subsequences in a string, 
  (2)~the expected total number of Lyndon subsequences in a string, 
  (3)~the expected number of distinct Lyndon subsequences in a string. 
\end{abstract}

\section{Introduction}

A string $x = uv$ is said to be a \emph{conjugate} of
another string $y$ if $y = vu$.
A string $w$ is called a \emph{Lyndon word}
if it is the lexicographically smallest among all of its conjugates.
It is also known that
$w$ is a Lyndon word iff $w$ is the lexicographically smallest suffix of itself
(excluding the empty suffix).

A \emph{factor} of a string $w$ is a sequence of
characters that appear contiguously in $w$.
A factor $f$ of a string $w$ is called a \emph{Lyndon factor}
if $f$ is a Lyndon word.
Lyndon factors enjoy a rich class of algorithmic and
stringology applications including:
counting and finding the maximal repetitions (a.k.a. runs) in a string~\cite{BannaiIINTT17} and in a trie~\cite{SugaharaNIBT19},
constant-space pattern matching~\cite{CrochemoreP91},
comparison of the sizes of
run-length Burrows-Wheeler Transform of a string and its reverse~\cite{GiulianiILPST21},
substring minimal suffix queries~\cite{BabenkoGKKS16},
the shortest common superstring problem~\cite{Mucha13},
and grammar-compressed self-index (Lyndon-SLP)~\cite{TsurutaKNIBT20}.

Since Lyndon factors are important combinatorial objects,
it is natural to wonder how many Lyndon factors can exist in a string.
Regarding this question,
the next four types of counting problems are interesting:
\begin{itemize}
\item $\TF(\sigma, n)$: the \emph{maximum total} number of Lyndon factors in a string of length $n$ over an alphabet of size $\sigma$.
\item $\DF(\sigma, n)$: the \emph{maximum} number of \emph{distinct} Lyndon factors in a string of length $n$ over an alphabet of size $\sigma$.
\item $\ETF(\sigma, n)$: the \emph{expected total} number of Lyndon factors in a string of length $n$ over an alphabet of size $\sigma$.
\item $\EDF(\sigma, n)$: the \emph{expected} number of \emph{distinct} Lyndon factors in a string of length $n$ over an alphabet of size $\sigma$.
\end{itemize}

Glen et al.~\cite{Glen2017counting} were the first who tackled
these problems, and they gave exact values
for $\DF(\sigma, n)$, $\ETF(\sigma, n)$, and $\EDF(\sigma, n)$.
Using the number $\lynnum(\sigma, n)$ of Lyndon words of length $n$ over an alphabet of size $\sigma$, their results can be written as shown in Table~\ref{tbl:Lyndon_factors}.
\begin{table}[!ht]
\begin{center}
\caption{The numbers of Lyndon factors in a string of length $n$ over an alphabet of size $\sigma$, where $n = m\sigma + p$ with $0 \leq p < \sigma$ for $\TF(\sigma, n)$ and $\DF(\sigma, n)$.}
\begin{tabular}{|l|l|l|}
\hline
\multicolumn{2}{|c|}{Number of Lyndon Factors in a String} \\ \hline \hline
Maximum Total $\TF(\sigma, n)$ & $\displaystyle \binom{n+1}{2} - (\sigma-p)\binom{m+1}{2} -p \binom{m+2}{2} +n$~[this work] \\ \hline
Maximum Distinct $\DF(\sigma, n)$ &  $\displaystyle \binom{n+1}{2} - (\sigma-p)\binom{m+1}{2} -p \binom{m+2}{2} +\sigma$~\cite{Glen2017counting} \\ \hline
Expected Total $\ETF(\sigma, n)$ & $\displaystyle \sum_{m=1}^{n} {L(\sigma,m) (n-m+1) \sigma ^{-m}}$~\cite{Glen2017counting} \\ \hline 
Expected Distinct $\EDF(\sigma, n)$ & $\displaystyle\sum_{m=1}^{n}  {L(\sigma,m) \sum_{s=1}^{\lfloor { n/m } \rfloor } { (-1)^{s+1} \binom{n-sm+s}{s} \sigma^{-sm}}}$~\cite{Glen2017counting} \\ \hline
\end{tabular}
\end{center}
\label{tbl:Lyndon_factors}
\end{table}

The first contribution of this paper is filling the missing piece
of Table~\ref{tbl:Lyndon_factors}, the exact value of $\TF(\sigma, n)$,
thus closing this line of research for Lyndon factors (substrings).

We then extend the problems to subsequences.
A subsequence of a string $w$ is a sequence of characters
that can be obtained by removing 0 or more characters from $w$.
A subsequence $s$ of a string $w$ is said to be a \emph{Lyndon subsequence}
if $s$ is a Lyndon word.
As a counterpart of the case of Lyndon factors,
it is interesting to consider the next four types of counting problems of Lyndon subsequences:
\begin{itemize}
\item $\maxnum(\sigma, n)$: the \emph{maximum total} number of Lyndon subsequences in a string of length $n$ over an alphabet of size $\sigma$.
\item $\distnum(\sigma, n)$: the \emph{maximum} number of \emph{distinct} Lyndon subsequences in a string of length $n$ over an alphabet of size $\sigma$.
\item $\exnum(\sigma, n)$: the \emph{expected total} number of Lyndon subsequences in a string of length $n$ over an alphabet of size $\sigma$.
\item $\exdistinct(\sigma, n)$: the \emph{expected} number of \emph{distinct} Lyndon subsequences in a string of length $n$ over an alphabet of size $\sigma$.
\end{itemize}

Among these, we present the exact values for
$\maxnum(\sigma, n)$, $\exnum(\sigma, n)$, and $\exdistinct(\sigma, n)$.
Our results are summarized in Table~\ref{tbl:Lyndon_subsequences}.
Although the main ideas of our proofs are analogous to the results for substrings,
there exist differences based on properties of substrings and subsequences.

\begin{table}[!ht]
\begin{center}
\caption{The numbers of Lyndon subsequences in a string of length $n$ over an alphabet of size $\sigma$, where $n = m\sigma + p$ with $0 \leq p < \sigma$ for $\maxnum(\sigma, n)$.}
\begin{tabular}{|l|l|l|}
\hline
\multicolumn{2}{|c|}{Number of Lyndon Subsequences in a String} \\ \hline \hline
Maximum Total $\maxnum(\sigma, n)$ & $\displaystyle 2^n-(p+\sigma)2^m+n+\sigma-1$~[this work] \\ \hline
Maximum Distinct $\distnum(\sigma, n)$ &  open \\ \hline
Expected Total $\exnum(\sigma, n)$ & $\displaystyle \sum_{m=1}^{n} \left[ \lynnum(\sigma,m) \binom{n}{m} \sigma ^{n-m} \right] \sigma^{-n} $~[this work] \\ \hline 
Expected Distinct $\exdistinct(\sigma, n)$ & $\displaystyle \sum_{m=1}^{n} \left[ \lynnum(\sigma,m) \sum_{k=m}^{n} { \binom{n}{k} (\sigma-1)^{n-k} } \right] \sigma^{-n}$~[this work] \\ \hline
\end{tabular}
\label{tbl:Lyndon_subsequences}
\end{center}
\end{table}

In the future work, we hope to determine the exact value for $\distnum(\sigma,n)$.
\section{Preliminaries} \label{sec:preliminaries}

\subsection{Strings}
Let $\Sigma = \{a_1, \ldots, a_{\sigma}\}$ be an ordered {\em alphabet}
of size $\sigma$ such that $a_1 < \ldots < a_{\sigma}$.
An element of $\Sigma^*$ is called a {\em string}.
The length of a string $w$ is denoted by $|w|$.
The empty string $\varepsilon$ is a string of length 0.
Let $\Sigma^+$ be the set of non-empty strings,
i.e., $\Sigma^+ = \Sigma^* - \{\varepsilon \}$.
The $i$-th character of a string $w$ is denoted by $w[i]$, where $1 \leq i \leq |w|$.
For a string $w$ and two integers $1 \leq i \leq j \leq |w|$,
let $w[i..j]$ denote the substring of $w$ that begins at position $i$ and ends at
position $j$. For convenience, let $w[i..j] = \varepsilon$ when $i > j$.
A string $x$ is said to be a subsequence of a string $w$ 
if there exists a set of positions $\{i_1, \ldots, i_{|x|}\}~(1 \leq i_1 < \ldots < i_{|x|} \leq |w|)$ 
such that $x = w[i_1] \cdots w[i_{|x|}]$.
We say that a subsequence $x$ occurs at $\{i_1, \ldots, i_{|x|}\}~(1 \leq i_1 < \ldots < i_{|x|} \leq |w|)$
if $x = w[i_1] \cdots w[i_{|x|}]$.

\subsection{Lyndon words}
A string $x = uv$ is said to be a \emph{conjugate} of
another string $y$ if $y = vu$.
A string $w$ is called a \emph{Lyndon word}
if it is the lexicographically smallest among all of its conjugates.
Equivalently, a string $w$ is said to be a Lyndon word,
if $w$ is lexicographically smaller than all of its non-empty proper suffixes.

Let $\mu$ be the \emph{M\"{o}bius function} on the set of positive integers defined as follows.
\[
	\mu(n) = \begin{cases}
	1 & (n=1) \\
	0 & (\mbox{if } n \mbox{ is divisible by a square}) \\
	(-1)^k & (\mbox{if } n \mbox{ is the product of } k \mbox{ distinct primes})
	\end{cases}
\]

It is known that the number $\lynnum(\sigma,n)$ of Lyndon words of length $n$ over an alphabet of size $\sigma$ 
can be represented as 
\begin{equation*}
  \lynnum(\sigma,n) = \frac{1}{n} \sum_{d \mid n}{ \mu\left(\frac{n}{d}\right)\sigma^d },
\end{equation*}
where $d|n$ is the set of divisors $d$ of $n$~\cite{Lothaire83}.
\section{Maximum total number of Lyndon subsequences}

Let $\maxnum(\sigma,n)$ be the maximum total number of Lyndon subsequences 
in a string of length $n$ over an alphabet $\Sigma$ of size $\sigma$.
In this section, we determine $\maxnum(\sigma,n)$.

\begin{theorem}
  For any $\sigma$ and $n$ such that $\sigma < n$,
  \[
    \maxnum(\sigma,n) = 2^n - (p + \sigma) 2^m + n + \sigma - 1
  \]
  where $n = m \sigma + p~(0 \leq p < \sigma)$.
  Moreover, 
  the number of strings that contain $\maxnum(\sigma,n)$ Lyndon subsequences is $\binom{\sigma}{p}$,
  and the following string $w$ is one of such strings;
  \[
    w = {a_1}^m \cdots {a_{\sigma-p}}^m {a_{\sigma-p+1}}^{m+1} \cdots {a_\sigma}^{m+1}.
  \]
\end{theorem}

\begin{proof}
  Consider a string $w$ of the form 
  \[
    w = {a_1}^{k_1}{a_2}^{k_2}\cdots{a_\sigma}^{k_\sigma}
  \]
  where $\sum_{i=1}^\sigma k_i = n$ and $k_i \geq 0$ for any $i$.
  For any subsequence $x$ of $w$,
  $x$ is a Lyndon word if $x$ is not a unary string of length at least 2.
  It is easy to see that this form is a necessary condition for the maximum number
  ($\because$ there exist several non-Lyndon subsequences if $w[i] > w[j]$ for some $i < j$).
  Hence, the number of Lyndon subsequences of $w$ can be represented as
  \begin{eqnarray*}
    (2^n-1) - \sum_{i=1}^{\sigma} {(2^{k_i}-1-k_i)} &=& 2^n - 1 - \sum_{i=1}^{\sigma}{2^{k_i}} + \sum_{i=1}^{\sigma}{k_i} + \sigma \\
                                              &=& 2^n-1-\sum_{i=1}^{\sigma} 2^{k_i} + n + \sigma.
  \end{eqnarray*}
  This formula is maximized when $\sum_{i=1}^{\sigma}{2^{k_i}}$ is minimized.
  It is known that 
  \[
    2^a + 2^b > 2^{a-1} + 2^{b+1}
  \]
  holds for any integer $a, b$ such that $a \geq b+2$.
  From this fact, $\sum_{i=1}^{\sigma}{2^{k_i}}$ is minimized
  when the difference of $k_i$ and $k_j$ is less than or equal to 1 for any $i, j$.
  Thus, if we choose $p~k_i$'s as $m+1$, and set $m$ for other $(\sigma-p)~k_i$'s 
  where $n = m \sigma + p~(0 \leq p < \sigma)$,
  then $\sum_{i=1}^{\sigma}{2^{k_i}}$ is minimized.
  Hence, 
  \begin{eqnarray*}
    \min(2^n-1-\sum_{i=1}^{\sigma} 2^{k_i} + n + \sigma) &=& 2^n-1 - p \cdot 2^{m+1} - (\sigma-p)2^m + n + \sigma \\
    &=& 2^n - (p+\sigma)2^m + n + \sigma - 1
  \end{eqnarray*}
  Moreover, one of such strings is 
  \[
    {a_1}^m \cdots {a_{\sigma-p}}^m {a_{\sigma-p+1}}^{m+1} \cdots {a_\sigma}^{m+1}.
  \]
  Therefore, this theorem holds.
\end{proof}

We can apply the above strategy to the version of substrings.
Namely, we can also obtain the following result.

\begin{corollary}
  Let $\TF(\sigma,n)$ be the maximum total number of Lyndon substrings in a string of length $n$
  over an alphabet of size $\sigma$.
  For any $\sigma$ and $n$ such that $\sigma < n$,
  \[
    \TF(\sigma,n) = \binom{n}{2} - (\sigma-p)\binom{m+1}{2} - p \binom{m+2}{2} + n
  \]
  where $n = m \sigma + p~(0 \leq p < \sigma)$.
  Moreover, 
  the number of strings that contain $\TF(\sigma,n)$ Lyndon subsequences is $\binom{\sigma}{p}$,
  and the following string $w$ is one of such strings;
  \[
    w = {a_1}^m \cdots {a_{\sigma-p}}^m {a_{\sigma-p+1}}^{m+1} \cdots {a_\sigma}^{m+1}.
  \]
\end{corollary}

\begin{proof}
  Consider a string $w$ of the form 
  \[
    w = {a_1}^{k_1}{a_2}^{k_2}\cdots{a_\sigma}^{k_\sigma}
  \]
  where $\sum_{i=1}^\sigma k_i = n$ and $k_i \geq 0$ for any $i$.
  In a similar way to the above discussion,
  the number of Lyndon substrings of $w$ can be represented as
  \[
    \binom{n+1}{2} - \sum_{i=1}^{\sigma} \left[ \binom{k_i+1}{2}-k_i \right ]
    = \binom{n+1}{2} - \sum_{i=1}^{\sigma} {\binom{k_i+1}{2}} +n.
  \]
  We can use the following inequation that holds for any $a, b$ such that $a \geq b+2$;
  \[
    \binom{a}{2}+\binom{b}{2} > \binom{a-1}{2} + \binom{b+1}{2}.
  \]
  Then,
  \[
    \min \left[ \binom{n+1}{2} - \sum_{i=1}^{\sigma} {\binom{k_i+1}{2}} + n \right] 
    = \binom{n}{2} - (\sigma-p)\binom{m+1}{2} - p \binom{m+2}{2} + n
  \]
  holds.
\end{proof}

Finally, we give exact values $\maxnum(\sigma,n)$ for several conditions
in Table~\ref{tbl:maximum-total}.

\begin{table}[ht]
  \begin{center}
    \caption{Values $\maxnum(\sigma,n)$ for $\sigma = 2,5,10$, $n = 1,2,\cdots,15$.}
    \begin{tabular}{|r|r|r|r|r|} \hline
      $n$ & $\maxnum(2,n)$ & $\maxnum(5,n)$ & $\maxnum(10,n)$ \\ \hline \hline
      1   &  1        & 1     & 1         \\
      2   &  3        & 3     & 3         \\
      3   &  6        & 7     & 7         \\
      4   & 13        & 15    & 15        \\
      5   & 26        & 31    & 31        \\
      6   & 55        & 62    & 63        \\
      7   & 122       & 125   & 127       \\
      8   & 233       & 252   & 255       \\
      9   & 474       & 507   & 511       \\
     10   & 971       & 1018  & 1023      \\
     11   & 1964      & 2039  & 2046      \\
     12   & 3981      & 4084  & 4093      \\
     13   & 8014      & 8177  & 8188      \\
     14   & 16143     & 16366 & 16379     \\
     15   & 32400     & 32747 & 32762     \\ \hline
    \end{tabular}
    \label{tbl:maximum-total}
  \end{center}
\end{table}
\section{Expected total number of Lyndon subsequences}

Let $\totalnum(\sigma,n)$ be the total number of Lyndon subsequences 
in all strings of length $n$ over an alphabet $\Sigma$ of size $\sigma$.
In this section, we determine the expected total number $\exnum(\sigma,n)$ of Lyndon subsequences 
in a string of length $n$ over an alphabet $\Sigma$ of size $\sigma$, 
namely, $\exnum(\sigma,n) = \totalnum(\sigma,n)/\sigma^n$.

\begin{theorem}
  For any $\sigma$ and $n$ such that $\sigma < n$,
  \[
    \totalnum(\sigma,n) = \sum_{m=1}^{n} \left[ \lynnum(\sigma,m) \binom{n}{m} \sigma ^{n-m} \right].
  \]
  Moreover, $\exnum(\sigma,n) = \totalnum(\sigma,n)/\sigma^n$.
\end{theorem}

\begin{proof}
  Let $\occ(w,x)$ be the number of occurrences of subsequence $x$ in $w$,
  and $L(\sigma,n)$ the set of Lyndon words of length less than or equal to $n$ over an alphabet of size $\sigma$.
  By a simple observation, $\totalnum(\sigma,n)$ can be written as
  \[
    \totalnum(\sigma,n) = \sum_{x\in\mathcal{L}(\sigma,n)} \sum_{w\in\Sigma^n} \occ(w,x).
  \]
  Firstly, we consider $\sum_{w\in\Sigma^n} \occ(w,x)$ for a Lyndon word $x$ of length $m$.
  Let $\{i_1, \ldots, i_m\}$ be a set of $m$ positions in a string of length $n$
  where $1 \leq i_1 < \ldots < i_m \leq n$.
  The number of strings that contain $x$ as a subsequence at $\{i_1, \ldots, i_m\}$ is $\sigma^{n-m}$.
  In addition, the number of combinations of $m$ positions is $\binom{n}{m}$.
  Hence, $\sum_{w\in\Sigma^n} \occ(w,x) = \binom{n}{m} \sigma^{n-m}$.
  This implies that
  \[
    \totalnum(\sigma,n) = \sum_{m=1}^{n} \left[ \lynnum(\sigma,m) \binom{n}{m} \sigma ^{n-m} \right].
  \]
  Finally, since the number of strings of length $n$ over an alphabet of size $\sigma$ is $\sigma^n$,
  $\exnum(\sigma,n) = \totalnum(\sigma,n)/\sigma^n$.
  Therefore, this theorem holds.
\end{proof}

Finally, we give exact values $\totalnum(\sigma,n), \exnum(\sigma,n)$ for several conditions
in Table~\ref{tbl:expect-total}.

\begin{table}[ht]
  \begin{center}
    \caption{Values $\totalnum(\sigma,n), \exnum(\sigma,n)$ for $\sigma=2,5, n=1,2,\cdots,10$.}
    \begin{tabular}{|r|r|r|r|r|} \hline
      $n$ & $\totalnum(2,n)$ & $\exnum(2,n)$ & $\totalnum(5,n)$     & $\exnum(5,n)$\\ \hline \hline
      1 & 2          & 1.00      & 5              & 1.00          \\
      2 & 9          & 2.25      &  60            & 2.40          \\
      3 & 32         & 4.00      &  565           & 4.52          \\
      4 & 107        & 6.69      &  4950          & 7.92          \\
      5 & 356        & 11.13     &  42499         & 13.60        \\
      6 & 1205       & 18.83     &  365050        & 23.36       \\
      7 & 4176       & 32.63     & 3163435        & 40.49        \\
      8 & 14798      & 57.80     & 27731650       & 70.99       \\
      9 & 53396      & 104.29    & 245950375      & 125.93     \\
     10 & 195323     & 190.75    & 2204719998     & 225.76     \\ \hline
    \end{tabular}
    \label{tbl:expect-total}
  \end{center}
\end{table}
\section{Expected number of distinct Lyndon subsequences}

Let $\totaldistinct(\sigma,n)$ be the total number of distinct Lyndon subsequences 
in all strings of length $n$ over an alphabet $\Sigma$ of size $\sigma$.
In this section, we determine the expected number $\exdistinct(\sigma,n)$ of distinct Lyndon subsequences 
in a string of length $n$ over an alphabet $\Sigma$ of size $\sigma$, 
namely, $\exdistinct(\sigma,n) = \totaldistinct(\sigma,n)/\sigma^n$.

\begin{theorem} \label{thm:expectdistinct}
  For any $\sigma$ and $n$ such that $\sigma < n$,
  \[
    \totaldistinct(\sigma,n) = \sum_{m=1}^{n} \left[ \lynnum(\sigma,m) \sum_{k=m}^{n} { \binom{n}{k} (\sigma-1)^{n-k} }  \right].
  \]
  Moreover, $\exdistinct(\sigma,n) = \totaldistinct(\sigma,n)/\sigma^n$.
\end{theorem}

To prove this theorem, we introduce the following lemmas.

\begin{lemma} \label{lem:subsec-occ}
  For any $x_1, x_2 \in \Sigma^m$ and $m, n~(m \leq n)$,
  the number of strings in $\Sigma^n$ which contain $x_1$ as a subsequence is equal to 
  the number of strings in $\Sigma^n$ which contain $x_2$ as a subsequence.
\end{lemma}
\begin{proof}[of Lemma~\ref{lem:subsec-occ}]
  Let $\subcontain(n,\Sigma,x)$ be the number of strings in $\Sigma^n$ 
  which contain a string $x$ as a subsequence.
  We prove $\subcontain(n,\Sigma,x_1) = \subcontain(n,\Sigma,x_2)$ for any $x_1, x_2 \in \Sigma^m$
  by induction on the length $m$.

  Suppose that $m=1$.
  It is clear that the set of strings which contain $x \in \Sigma$ is $\Sigma^n - (\Sigma - \{x\})^n$,
  and $\subcontain(n,\Sigma,x) = \sigma^n - (\sigma - 1)^n$.
  Thus, $\subcontain(n,\Sigma,x_1) = \subcontain(n,\Sigma,x_2)$ for any $x_1, x_2$ 
  if $|x_1| = |x_2| = 1$.

  Suppose that the statement holds for some $k \geq 1$.
  We prove $\subcontain(n,\Sigma,x_1) = \subcontain(n,\Sigma,x_2)$ for any $x_1, x_2 \in \Sigma^{k+1}$ 
  by induction on $n$.
  If $n = k+1$, then $\subcontain(n,\Sigma,x_1) = \subcontain(n,\Sigma,x_2) = 1$.
  Assume that the statement holds for some $\ell \geq k+1$.
  Let $x = yc$ be a string of length $k+1$ such that $y \in \Sigma^k, c \in \Sigma$.
  Each string $w$ of length $\ell+1$ which contains $x$ as a subsequence satisfies either 
  \begin{itemize}
    \item $w[1..\ell]$ contains $x$ as a subsequence, or
    \item $w[1..\ell]$ does not contain $x$ as a subsequence.
  \end{itemize}
  The number of strings $w$ in the first case is $\sigma \cdot \subcontain(j,\Sigma,yc)$.
  On the other hand, the number of strings $w$ in the second case 
  is $\subcontain(\ell,\Sigma,y) - \subcontain(\ell,\Sigma,yc)$.
  Hence, $\subcontain(\ell+1,\Sigma,x) = \sigma \subcontain(\ell,\Sigma,yc) + \subcontain(\ell,\Sigma,y) - \subcontain(\ell,\Sigma,yc)$.
  Let $x_1 = y_1 c_1$ and $x_2 = y_2 c_2$ be strings of length $k+1$.
  By an induction hypothesis, $\subcontain(\ell,\Sigma,y_1 c_1) = \subcontain(\ell,\Sigma,y_2 c_2)$ and 
  $\subcontain(\ell,\Sigma,y_1) = \subcontain(\ell,\Sigma,y_2)$ hold.
  Thus, $\subcontain(\ell+1,\Sigma,x_1) = \subcontain(\ell+1,\Sigma,x_2)$ also holds.
  
  Therefore, this lemma holds.
\end{proof}

\begin{lemma} \label{lem:subsec-occ-num}
  For any string $x$ of length $m \leq n$,
  \[
    \subcontain(n,\Sigma,x) = \sum_{k=m}^{n}{ \binom{n}{k} (\sigma-1)^{n-k} }.
  \]
\end{lemma}

\begin{proof}[of Lemma~\ref{lem:subsec-occ-num}]
  For any character $c$, 
  it is clear that the number of strings that contain $c$ exactly $k$ times is $\binom{n}{k}(\sigma-1)^{n-k}$.
  By Lemma~\ref{lem:subsec-occ}, 
  \[
    \subcontain(n,\Sigma,x) = \subcontain(n,\Sigma,c^m) = \sum_{k=m}^{n}{ \binom{n}{k} (\sigma-1)^{n-k} }.
  \]
  Hence, this lemma holds.
\end{proof}

Then, we can obtain Theorem~\ref{thm:expectdistinct} as follows.

\begin{proof}[of Theorem~\ref{thm:expectdistinct}]
  Thanks to Lemma~\ref{lem:subsec-occ-num},
  the number of strings of length $n$ which contain a Lyndon word of length $m$ 
  is also $\sum_{k=m}^{n}{ \binom{n}{k} (\sigma-1)^{n-k} }$.
  Since the number of Lyndon words of length $m$ over an alphabet of size $\sigma$ is $\lynnum(\sigma,m)$,
  \[
    \totaldistinct(\sigma,n) = \sum_{m=1}^{n} \left[ \lynnum(\sigma,m) \sum_{k=m}^{n} { \binom{n}{k} (\sigma-1)^{n-k} }  \right].
  \]
  Finally, since the number of strings of length $n$ over an alphabet of size $\sigma$ is $\sigma^n$,
  $\exdistinct(\sigma,n) = \totaldistinct(\sigma,n)/\sigma^n$.
  Therefore, Theorem~\ref{thm:expectdistinct} holds.
\end{proof}

We give exact values $\exdistinct(\sigma,n)$ 
for several conditions in Table~\ref{tbl:expect-distinct}.

\begin{table}[ht]
  \caption{Values $\exdistinct(\sigma,n)$ for $\sigma=2,5, n=1,\dots,10,15,20$.}
  \begin{center}
    \begin{tabular}{|r|r|r|} \hline
      $n$ & $\exdistinct(2,n)$ & $\exdistinct(5,n)$\\ \hline \hline
      1 & 1.00   & 1.00     \\
      2 & 1.75   & 2.20     \\
      3 & 2.50   & 3.80     \\
      4 & 3.38   & 6.09     \\
      5 & 4.50   & 9.51     \\
      6 & 6.00   & 14.80    \\
      7 & 8.03   & 23.12    \\
      8 & 10.81  & 36.43    \\
      9 & 14.63  & 57.95    \\
     10 & 19.93  & 93.08    \\ 
     15 & 100.57 & 1121.29  \\
     20 & 559.42 & 15444.90 \\ \hline
    \end{tabular} \label{tbl:expect-distinct}
  \end{center}
\end{table}
\section*{Acknowledgments}
This work was supported by JSPS KAKENHI Grant Numbers 
JP18K18002 (YN), JP21K17705 (YN), 
JP18H04098 (MT), 
JST ACT-X Grant Number JPMJAX200K (YN),
and by JST PRESTO Grant Number JPMJPR1922 (SI).

\bibliographystyle{abbrv}
\bibliography{ref}

\end{document}